\documentclass[leqno, 12pt]{amsart}

\usepackage{latexsym, amsmath,amssymb,amsfonts, graphics, mathrsfs}
\usepackage[all]{xy}

\theoremstyle{plain} 
\newtheorem{theorem}{\indent\sc Theorem}[section]
\newtheorem{lemma}[theorem]{\indent\sc Lemma}
\newtheorem{corollary}[theorem]{\indent\sc Corollary}
\newtheorem{proposition}[theorem]{\indent\sc Proposition}

\theoremstyle{definition} 
\newtheorem{definition}[theorem]{\indent\sc Definition}
\newtheorem{remark}[theorem]{\indent\sc Remark}

\begin{document}

\title[Almost Ricci-flat 5-manifolds]{Kummer-type constructions of almost Ricci-flat 5-manifolds}

\author{Chanyoung Sung}



%


\address{Dept. of mathematics education\endgraf
Korea National University of Education\endgraf
Cheongju\endgraf
28173\endgraf
Republic of Korea
}
\email{cysung@kias.re.kr}


\maketitle

\begin{abstract}
A smooth closed manifold $M$ is called almost Ricci-flat if $$\inf_g||\textrm{Ric}_g||_\infty\cdot \textrm{diam}_g(M)^2=0$$
where $\textrm{Ric}_g$ and $\textrm{diam}_g$ respectively denote the Ricci tensor and the diameter of $g$ and $g$ runs over all Riemannian metrics on $M$.
By using Kummer-type method we construct a smooth closed almost Ricci-flat nonspin 5-manifold $M$ which is simply connected.
It's minimal volume vanishes, namely it collapses with sectional curvature bounded.
\end{abstract}

\setcounter{section}{0}
\setcounter{equation}{0}

\section{Introduction}

M. Gromov \cite{gromov-1} considered an almost flat manifold $M$ which is defined as a smooth closed manifold satisfying
$$\inf_g||\textrm{Rm}_g||_\infty\cdot \textrm{diam}_g(M)^2=0$$ where $\textrm{Rm}_g$ and $\textrm{diam}_g$ respectively denote the Riemann curvature tensor and the diameter of $g$ and $g$ runs over all smooth Riemannian metrics on $M$. Such a manifold has to be a collapsing manifold and finitely covered by a compact nilmanifold.
Moreover a volume-noncollapsed almost flat manifold is actually flat, more precisely given $n\in \Bbb Z^+$ and $v > 0$ there is some $\varepsilon= \varepsilon(n, v) > 0$ so that
any Riemannian $n$-manifold $(M, g)$ with $\textrm{Vol}_g(M) \geq v, \textrm{diam}_g(M)\leq 1$, and $|\textrm{Rm}_g|\leq \varepsilon$
where $\textrm{Vol}_g$ denotes the volume of $g$ admits a flat metric.

In an analogous way V. Kapovitch and J. Lott \cite{Lott1, Lott2} considered almost Ricci-flat manifolds with the aim of obtaining a Ricci-flat manifold.
A smooth closed manifold $M$ is called almost Ricci-flat if $\mu(M)=0$ for $$\mu(M):=\inf_g||\textrm{Ric}_g||_\infty\cdot \textrm{diam}_g(M)^2$$
where $\textrm{Ric}_g$ denotes the Ricci tensor of $g$ and $g$ runs over all smooth Riemannian metrics on $M$.
They found some topological conditions to ensure that a volume-noncollapsed almost Ricci-flat manifold admits a Ricci-flat metric.

By the Cheeger-Gromoll splitting theorem \cite{CG} any smooth closed Ricci-flat manifold must be finitely covered by the product of a flat $m$-torus $T^m$ and
a compact simply connected Ricci-flat manifold. According to Berger's classification \cite{ber, peter},
an irreducible simply connected Ricci-flat $n$-manifold is either locally symmetric or has holonomy $SU(\frac{n}{2})$ or $Sp(\frac{n}{4})$ or $G_2$ for $n=7$ or
$\rm Spin(7)$ for $n=8$, or $SO(n)$. In case that it is locally symmetric, it must be flat, and there have been found many examples of compact simply connected Ricci-flat manifolds with special holonomy.

There are also examples of almost Ricci-flat manifolds which do not admit a Ricci-flat metric. M. Anderson \cite{ander} constructed such examples in dimension 4 by performing
surgeries along circles of $T^4$ with $(D^2\times S^2,g_{S})$ where $g_{S}$ is the Riemannian Schwarzschild metric.
While it is perhaps remote from the complete understanding of almost Ricci-flat manifolds, it is worthwhile to have different types of examples of almost Ricci-flat manifolds at hand.

Gluing techniques are major tools to construct geometric structures in differential geometry.
In particular Kummer-type construction has been useful for constructing simply connected Ricci-flat manifolds with special holonomy such as hyper-K\"ahler $K3$ surfaces \cite{topi1, topi2, LS}, Joyce's $G_2$-manifolds \cite{joyce1}, and Joyce's $Spin(7)$-manifolds \cite{joyce2}.
S. Brendle and N. Kapouleas \cite{bren} used this method to construct simply connected almost Ricci-flat 4-manifolds.
In this article we use Kummer-type construction to construct simply connected almost Ricci-flat 5-manifolds.  More precisely we prove
\begin{theorem}\label{main-thm}
There exists a smooth closed almost Ricci-flat nonspin 5-manifold $X$ which is simply connected and has a nontrivial torsion subgroup of $H_2(X,\Bbb Z)$.
A sequence of metrics on $X$ realizing $\mu(X)=0$ has a positive lower bound of volumes, but the minimal volume $\textrm{MinVol}(X)$ of $X$ vanishes so that $X$ collapses
with sectional curvature bounded.
\end{theorem}
The \emph{minimal volume} introduced by Gromov \cite{gromov} is defined as $$\textrm{MinVol}(X):=\inf_g \{\textrm{Vol}_g(X)|\ |K_g|\leq 1\}$$
where $K_g$ denotes the sectional curvature of $g$ and $g$ runs over all Riemannian metrics on $X$.
The vanishing of $\textrm{MinVol}(X)$ implies that all the characteristic numbers of $X$ are zero via Chern-Weil theory
and Gromov's \emph{simplicial volume} $$||X||:=\inf \{\Sigma_i |r_i| \ | \ r_i\ \textrm{are the coefficients of a real cycle representing}\ [X]\}$$ of $X$ is also zero
by the inequality $$||X|| \leq (n-1)^nn!\ \textrm{MinVol}(X)$$ where $n=\dim (X)$.

Like other Kummer-type constructions, this example may shed light on the study of canonical geometry in dimension 5.
However, for now we do not know whether it admits a Ricci-flat metric or not.

\section{Construction of $X$}
We start with a flat $5$-dimensional torus $T^5=\Bbb R^5/ \Bbb Z^5$ where $$\Bbb Z^5=\{(x_1,\cdots,x_5)|x_i\in \Bbb Z\}.$$
For convenience we always express a point of $T^5$ as $x=(x_1,\cdots,x_5)\in \Bbb R^5$ modulo $\Bbb Z^5$. For the following 3 isometric involutions
$$\alpha : x \mapsto (x_1,-x_2,-x_3,\frac{1}{2}-x_4,-x_5),$$
$$\beta : x \mapsto (-x_1,\frac{1}{2}-x_2,-x_3,x_4,-x_5),$$
$$\gamma : x \mapsto (-x_1,-x_2,\frac{1}{2}-x_3,-x_4,x_5)$$ of $T^5$,
one can easily check that they all commute and hence generate $$\Gamma:=(\Bbb Z_2)^3=\langle \alpha \rangle \oplus \langle \beta \rangle \oplus \langle \gamma \rangle.$$
The fixed point sets of $\alpha$ are $16$ copies of circle $$S_\alpha:=S^1\times \{(p_2,\cdots,p_5)\}$$
where $p_2,p_3,p_5$ are either 0 or $\frac{1}{2}$ and $p_4$ is either $\frac{1}{4}$ or $\frac{3}{4}$, and similarly we have 16 copies of $S_\beta$ and  16 copies of $S_\gamma$.
All together they are disjoint 48 circles which we shall denote by
$$Z:=\textrm{the union of 48 singular circles}$$
and these are all the singular locus of the $\Gamma$ action. Indeed
$$\alpha\beta : x \mapsto (-x_1,\frac{1}{2}+x_2,x_3,\frac{1}{2}-x_4,x_5),$$
$$\beta\gamma : x \mapsto (x_1,\frac{1}{2}+x_2,\frac{1}{2}+x_3,-x_4,-x_5),$$
$$\alpha\gamma : x \mapsto (-x_1,x_2,\frac{1}{2}+x_3,\frac{1}{2}+x_4,-x_5),$$
and
$$\alpha\beta\gamma : x \mapsto (x_1,\frac{1}{2}-x_2,\frac{1}{2}-x_3,\frac{1}{2}+x_4,x_5),$$ act without any fixed point.
Moreover $\langle\alpha,\beta\rangle$ acts freely on the set of 16 $S_\gamma$ identifying them into 4 $S_\gamma$.
Similarly $\langle\beta,\gamma\rangle$ identifies 16 $S_\alpha$ to 4 of them, and $\langle\alpha,\gamma\rangle$ identifies 16 $S_\beta$ to 4 of them. Thus the singular set of
$T^5/\Gamma$ is a disjoint union of 12 copies of $S^1$, and the singularity at each $S^1$ is modeled on $S^1\times (\Bbb C^2/\{\pm 1\})$.

To resolve these singularities, we delete a neighborhood $S^1\times (B_{\frac{1}{10}}/\{\pm 1\})$ of each singular $S^1$
where $B_{\frac{1}{10}}$ is a ball of radius $\frac{1}{10}$ around $0\in \Bbb C^2$, and then graft $S^1\times Y$ along each boundary component $S^1\times \Bbb RP^3$,
where $Y$ is a disk bundle of $T^*S^2$. The resulting manifold is our desired manifold $X$.
Let's set $$\breve{T}:=T^5/\Gamma-U$$ where $U$ is the union of 12 neighborhoods $S^1\times (B_{\frac{1}{10}}/\{\pm 1\})$.

There is another description of $X$ viewed as the $\Gamma$-quotient of a ``blow-up" of $T^5$ along $Z$.
For the quotient map $$\pi : T^5\rightarrow T^5/\Gamma,$$ the boundary of $\pi^{-1}(\breve{T})$ consists of 48 $S^1\times S^3$,
so one can glue 48 $S^1\times \hat{Y}$ along the boundary, where $\hat{Y}$ is the associated disk bundle of the Hopf fibration over $S^2$.
(Indeed $\hat{Y}$ is the algebro-geometric blow-up of 4-ball $B^4$ at the origin, where the origin is replaced by a sphere of self-intersection $-1$.)
By the ``blow-up" of $T^5$ we mean this resulting manifold which we denote by $\hat{T}^5$. Since the $\Bbb Z_2$-quotient of $\hat{Y}$ is $Y$,
the $\Gamma$-action on $T^5$ can be (smoothly) extended to $\hat{T}^5$ with the quotient space equal to $X$.
For the later purpose, we denote this quotient map by $$\hat{\pi}: \hat{T}^5\rightarrow X.$$

Now we will construct the metric on $X$. First note that $\breve{T}$ has a flat metric coming from the flat orbifold metric on $T^5/\Gamma$,
and we dilate this metric by multiplying $(20d)^2$ for a constant $d\gg 1$. $T^*S^2$ has an ALE Ricci-flat K\"ahler metric known as the Eguchi-Hanson metric
$$g_{\textrm{EH}}=\frac{dr^2}{1-r^{-4}}+r^2(1-r^{-4})\sigma_3^2+r^2(\sigma_1^2+\sigma_2^2)$$ for $r\in [1,\infty)$ and
the left-invariant coframe $\{\sigma_1,\sigma_2,\sigma_3\}$ of $S^3$ satisfying $$d\sigma_i=2\varepsilon_{ijk}\sigma_j\wedge\sigma_k.$$
It asymptotically approaches the Euclidean metric so that there exists a constant $C_I$ such that outside a neighborhood of the zero section,
\begin{eqnarray}\label{metric-bound1}
|\frac{\partial^{|I|}}{\partial x_{I}}((g_{\textrm{EH}})_{\mu\nu}-\delta_{\mu\nu})|\leq \frac{C_I}{|x|^{4+|I|}}
\end{eqnarray}
for all $\mu,\nu=1,\cdots,5$, where $I:=(i_1,\cdots,i_n)$ denotes a multi-index with all $i_j\geq 0$.
By the curvature formula
\begin{eqnarray}\label{Sinai-mountain1}
R^l_{ijk}=-\partial_j\Gamma^l_{ik}+\partial_i\Gamma^l_{jk}-\Gamma^m_{ik}\Gamma^l_{jm}+\Gamma^m_{jk}\Gamma^l_{im}
\end{eqnarray}
where
$$\Gamma^k_{ij}=\frac{1}{2}g^{kl}(\partial_ig_{jl}+\partial_jg_{il}-\partial_lg_{ij}),$$
the curvature $R^l_{ijk}$ of $g_{\textrm{EH}}$ decays to zero as $O(r^{-6})$.

Take any smooth cutoff function $\rho_1(r) : \Bbb R^+ \rightarrow \Bbb R^+\cup \{0\}$ satisfying
$$
\rho_1(r) =
\left\{
\begin{array}{ll}  1
   &\mbox{for } r\leq 1\\
  0 &\mbox{for } r\geq 2.
\end{array}\right.
$$
Then our desired cutoff function is $\rho_d(r):=\rho_1(\frac{r}{d})$ which then satisfies
$$
\rho_d(r) =
\left\{
\begin{array}{ll}  1
   &\mbox{for } r\leq d\\
  0 &\mbox{for } r\geq 2d,
\end{array}\right.
$$
and $$|D^m\rho_d|\leq \frac{c_m}{d^m}$$ for any positive integer $m$ where $c_m$ is a constant depending on $m$.

Consider a metric $$\tilde{g}_{d}:=\rho_d\ g_{\textrm{EH}}+(1-\rho_d)g_{\textrm{E}}$$ on $T^*S^2$, which is equal to the Euclidean metric $g_{\textrm{E}}$ for $r\geq 2d$, and put this metric on $Y$ using a diffeomorphism to $T^*S^2$.
We now glue the metrics on $\breve{T}$ and $S^1(20d)\times Y$ where $S(20d)$ denotes the circle of length 20d to get a smooth metric $g_d$ on $X$.
The important property of $g_d$ is that there exists a constant $\tilde{C}_I$ independent of $d$ such that
\begin{eqnarray}\label{metric-bound2}
|\frac{\partial^{|I|}}{\partial x_{I}}((g_d)_{\mu\nu}-\delta_{\mu\nu})|\leq \frac{\tilde{C}_I}{d^{4+|I|}}
\end{eqnarray}
on the gluing regions for all $\mu,\nu=1,\cdots,5$ and hence $|\textrm{Ric}_{{g}_d}|$ is $O(\frac{1}{d^6})$ by (\ref{Sinai-mountain1}).

By shrinking back $(X,g_d)$ by the factor $(\frac{1}{20d})^2$, we get a desired metric $g$ on $X$ satisfying
 $$||\textrm{Ric}_{g}||_\infty\leq \frac{C_1}{d^4},\ \ \textrm{Vol}_g(X)\geq C_2,\ \ \textrm{diam}_g(X)\in [1,1+C_3]$$ for constants $C_i>0$ independent of $d$.
Taking $d$ arbitrarily large, one can achieve the almost Ricci-flatness on $X$.

\section{Topology of $X$}
In this section any path and loop are meant to be continuous.
\begin{lemma}\label{Ha-Seong}
For any base point $p_0\in T^5$ and any loop $\tau:[0,1]\rightarrow T^5$ based at $p_0$, $\pi(\tau)$ is homotopic to the constant loop at $\pi(p_0)$.
\end{lemma}
\begin{proof}
One can take generators $[\rho_1],\cdots,[\rho_5]$ of $\pi_1(T^5,p_0)\simeq \Bbb Z^5$ such that each $\rho_i$ is parallel to $x_i$-axis, meaning that
$$\rho_i(t)=p_0+te_i\ \ \ \textrm{mod}\ \Bbb Z^5$$ for $t\in [0,1]$ where $\{e_1,\cdots,e_5\}$ is the standard basis of $\Bbb R^5$.
Then there exists $\rho_{i_1},\cdots,\rho_{i_k}$ and $\epsilon_1,\cdots,\epsilon_k\in \{\pm 1\}$ so that $[\tau]=[\rho_{i_1}^{\epsilon_1}\cdot\ \cdots\ \cdot\rho_{i_k}^{\epsilon_k}]$ where $(\cdot)^{-1}$ denotes the inverse path.

It is enough to show that each $\pi(\rho_i)$ is homotopic to the constant loop at $\pi(p_0)$. First let's see the case of $\pi(\rho_1)$.
Note that the loop $\rho_1$ is freely homotopic to the loop $\bar{\rho}_1$ based at $p_1=(0,\frac{1}{4},0,0,0)\in T^5$ defined by
$$\bar{\rho}_1(t)=p_1+te_1\ \ \ \textrm{mod}\ \Bbb Z^5$$ for $t\in [0,1]$. (The translation by $s(p_1-p_0)$ for $s\in [0,1]$ gives the free homotopy of two loops.)
The reason for this particular choice of $p_1$ will become obvious right next.

Since $$\beta(\bar{\rho}_1(t))=p_1-te_1\equiv p_1+(1-t)e_1\ \ \ \textrm{mod}\ \Bbb Z^5,$$ $\pi(\bar{\rho}_1(t))=\pi(p_1+(1-t)e_1)$. Then  writing $\pi(\bar{\rho}_1)$ as
$$
\pi(\bar{\rho}_1(t)) =
\left\{
\begin{array}{ll} \pi(p_1+te_1)
   &\mbox{for } t\in [0,\frac{1}{2}]\\
   \pi(p_1+(1-t)e_1) &\mbox{for } t\in [\frac{1}{2},1],
\end{array}\right.
$$
it is homotopic to the constant loop at $\pi(p_1)$ by the lemma below.
Thus $\pi(\rho_1)$ is freely homotopic to the constant loop at $\pi(p_1)$, and this implies that $[\pi(\rho_1)]=0\in \pi_1(T^5/\Gamma, \pi(p_0))$.

Other cases can be verified in the same way. For instance, for the case of $\pi(\rho_2)$, instead of $p_1$ one can use $p_2=(0,0,0,\frac{1}{4},0)$ fixed by $\alpha$
which inverts the direction of $\bar{\rho}_2$.
\end{proof}

We name the following elementary fact as an lemma for convenience.
\begin{lemma}\label{Ha-Seong-9}
Let $M$ be any topological space and $\tau:[0,\frac{1}{2}]\rightarrow M$ be a path. Then the loop $\bar{\tau}:[0,1]\rightarrow M$ defined as
$$
\bar{\tau}(t)=
\left\{
\begin{array}{ll} \tau(t)
   &\mbox{for } t\in [0,\frac{1}{2}]\\
   \tau(1-t) &\mbox{for } t\in [\frac{1}{2},1]
\end{array}\right.
$$
is homotopic to the constant loop at $\tau(0)$.
\end{lemma}

\begin{proposition}
$X$ is simply connected.
\end{proposition}
\begin{proof}
We shall first show that $\pi_1(T^5/\Gamma, \pi(p_0))=\{0\}$ where $p_0$ is chosen from $T^5-Z$.
Specifically we choose $p_0$ to be $(0,0,0,\frac{1}{4}+\epsilon,0)$ for $\epsilon\in (0,\frac{1}{100})$ near one of $S_\alpha$ which is precisely
$$S_\alpha^1:=S^1\times \{(0,0,\frac{1}{4},0)\}.$$

Let $[\sigma]$ be any element in $\pi_1(T^5/\Gamma, \pi(p_0))$. Since $S^1\times \{pt\}$ is a strong deformation retract of $S^1\times (B_{\frac{1}{10}}/\{\pm 1\})$ in an obvious way,
one can use this deformation to choose a loop $\sigma:[0,1]\rightarrow T^5/\Gamma$ representing $[\sigma]$ such that $\sigma$ dose not intersect with $\pi(Z)$.
Since $$\pi:T^5-Z\rightarrow T^5/\Gamma-\pi(Z)$$ is a covering map, one can lift $\sigma$ to $\tilde{\sigma}:[0,1]\rightarrow T^5-Z$ such that $\tilde{\sigma}(0)=p_0$.

There are two possibilities for $\tilde{\sigma}$, whether it is a loop or a path with $\tilde{\sigma}(0)\ne \tilde{\sigma}(0)$.
In the former case, by Lemma \ref{Ha-Seong} $\sigma=\pi(\tilde{\sigma})$ must be homotopically trivial.

Now we deal with the latter case which needs an analysis of further subcases.
There are $|\Gamma|=8$ points in $\pi^{-1}(\pi(p_0))$. They are all in the $\epsilon$ distance away from $\pi^{-1}(\pi(S_\alpha^1))$
which is the union of $$S_\alpha^1,\  S_\alpha^2:=\beta(S_\alpha^1),\  S_\alpha^3:=\gamma(S_\alpha^1),\  S_\alpha^4:=\beta\gamma(S_\alpha^1).$$
(Recall that $\langle\beta,\gamma\rangle$ identifies 16 $S_\alpha$ to 4 of them.)
We label the points of $\pi^{-1}(\pi(p_0))$ as $q_1,\cdots,q_8$ such that
$$q_1=p_0,\ q_2=\alpha(p_0),\ q_3=\beta(p_0),\ q_4=\beta\alpha(p_0),$$
$$q_5=\gamma(p_0),\ q_6=\gamma\alpha(p_0),\ q_7=\beta\gamma(p_0),\ q_8=\beta\gamma\alpha(p_0).$$
Observe that $q_{2i-1},q_{2i}$ are away from $S_\alpha^i$ by the $\epsilon$ distance.
We have 7 subcases according to where $\tilde{\sigma}(1)$ lands.

In the 1st case when $\tilde{\sigma}(1)=q_2$, we take a path $\hat{\sigma}_2:[0,1]\rightarrow T^5$ such that $\hat{\sigma}_2(0)=p_0, \hat{\sigma}_2(1)=q_2$
and $\textrm{Im}(\hat{\sigma}_2)$ lies on a round 3-sphere $$\{x=(0,x_2,\cdots,x_5)\in T^5|\ \ ||x-(0,0,0,\frac{1}{4},0)||=\epsilon\}.$$
(Observe that $q_1$ and $q_2$ are antipodal to each other on this 3-sphere.)
Then by Lemma \ref{Ha-Seong} $\pi(\tilde{\sigma}\cdot \hat{\sigma}_2^{-1})=\sigma\cdot (\pi(\hat{\sigma}_2))^{-1}$ is homotopic to constant so that
$[\sigma]=[\pi(\hat{\sigma}_2)]\in \pi_1(T^5/\Gamma,\pi(p_0))$. Note that $\pi(\hat{\sigma}_2)$ is placed in a 4-orbifold $\{pt\}\times (B_{\frac{1}{10}}/\{\pm 1\})$
which is contractible, so
\begin{eqnarray}\label{HSKim-99}
[\pi(\hat{\sigma}_2)]=0.
\end{eqnarray}
Therefore we proved that $[\sigma]=0$ in this first case.

In the 2nd case when $\tilde{\sigma}(1)$ is $q_3$, we choose any point $q\in T^5$ which is fixed by $\beta$ and take any path $\check{\sigma}:[0,\frac{1}{2}]\rightarrow T^5$
such that $\check{\sigma}(0)=p_0, \check{\sigma}(\frac{1}{2})=q$. Define a path $\hat{\sigma}_3:[0,1]\rightarrow T^5$ as
$$
\hat{\sigma}_3(t) =
\left\{
\begin{array}{ll} \check{\sigma}(t)
   &\mbox{for } t\in [0,\frac{1}{2}]\\
   \beta(\check{\sigma}(1-t)) &\mbox{for } t\in [\frac{1}{2},1].
\end{array}\right.
$$
Then again by Lemma \ref{Ha-Seong} $\pi(\tilde{\sigma}\cdot \hat{\sigma}_3^{-1})=\sigma\cdot (\pi(\hat{\sigma}_3))^{-1}$ is homotopic to a constant so that
\begin{eqnarray*}
[\sigma]&=&[\pi(\hat{\sigma}_3)]\\
&=& 0
\end{eqnarray*}
where the second equality is due to Lemma \ref{Ha-Seong-9}, finishing in the 2nd case.

In the 3rd case when $\tilde{\sigma}(1)$ is $q_4$, this time we take $\check{\sigma}:[0,\frac{1}{2}]\rightarrow T^5$ such that
$\check{\sigma}(0)=q_2, \check{\sigma}(\frac{1}{2})=q$ where $q$ was as above, and define a path $\hat{\sigma}_4:[0,1]\rightarrow T^5$ in the same way as $\hat{\sigma}_3$
with this new $\check{\sigma}$.
Then again by Lemma \ref{Ha-Seong} $\pi(\tilde{\sigma}\cdot (\hat{\sigma}_2\cdot\hat{\sigma}_4)^{-1})=\sigma\cdot (\pi(\hat{\sigma}_2)\cdot\pi(\hat{\sigma}_4))^{-1}$ is
homotopic to a constant so that
\begin{eqnarray*}
[\sigma]&=&[\pi(\hat{\sigma}_2)\cdot\pi(\hat{\sigma}_4)]\\
&=& 0
\end{eqnarray*}
where the second equality is due to (\ref{HSKim-99}) and Lemma \ref{Ha-Seong-9}, finishing the 3rd case.

In the 4th case when $\tilde{\sigma}(1)$ is $q_5$, the proof of $[\sigma]=0$ is almost verbatim to the above 2nd case except that $\beta$ is replaced by $\gamma$ and
one has to use a point $q'\in T^5$ which is fixed by $\gamma$.

In the 5th case when $\tilde{\sigma}(1)$ is $q_6$, the proof of $[\sigma]=0$ is almost verbatim to the 3rd case by using $\gamma$ and $q'\in T^5$ chosen in the 4th case.

In the 6th case when $\tilde{\sigma}(1)$ is $q_7$, an additional treatment is needed because $\beta\gamma=\gamma\beta$ which enters into this case has no fixed points.
Using $q'\in T^5$ chosen above and choosing any path $\breve{\sigma}:[0,\frac{1}{2}]\rightarrow T^5$ such that $\breve{\sigma}(0)=q_3, \breve{\sigma}(\frac{1}{2})=q'$,
define $\hat{\sigma}_7:[0,1]\rightarrow T^5$ as
$$
\hat{\sigma}_7(t) =
\left\{
\begin{array}{ll} \breve{\sigma}(t)
   &\mbox{for } t\in [0,\frac{1}{2}]\\
   \gamma(\breve{\sigma}(1-t)) &\mbox{for } t\in [\frac{1}{2},1].
\end{array}\right.
$$
Then by Lemma \ref{Ha-Seong} $\pi(\tilde{\sigma}\cdot (\hat{\sigma}_3\cdot\hat{\sigma}_7)^{-1})=\sigma\cdot (\pi(\hat{\sigma}_3\cdot\hat{\sigma}_7))^{-1}$ is homotopic to constant so that
\begin{eqnarray*}
[\sigma]&=&[\pi(\hat{\sigma}_3)\cdot \pi(\hat{\sigma}_7)]\\
&=& 0
\end{eqnarray*}
where Lemma \ref{Ha-Seong-9} is used twice, finishing the 6th case.

In the last 7th case when $\tilde{\sigma}(1)$ is $q_8$, this time we choose $\breve{\sigma}:[0,\frac{1}{2}]\rightarrow T^5$ such that
$\breve{\sigma}(0)=q_4, \breve{\sigma}(\frac{1}{2})=q'$, and define $\hat{\sigma}_8:[0,1]\rightarrow T^5$ in the same way as $\hat{\sigma}_7$ with this new $\breve{\sigma}$.
Then by Lemma \ref{Ha-Seong} $\pi(\tilde{\sigma}\cdot (\hat{\sigma}_2\cdot\hat{\sigma}_4\cdot\hat{\sigma}_8)^{-1})=\sigma\cdot (\pi(\hat{\sigma}_2)\cdot\pi(\hat{\sigma}_4)\cdot\pi(\hat{\sigma}_8))^{-1}$ is homotopic to a constant so that
\begin{eqnarray*}
[\sigma]&=&[\pi(\hat{\sigma}_2)\cdot\pi(\hat{\sigma}_4)\cdot\pi(\hat{\sigma}_8)]\\
&=& 0
\end{eqnarray*}
where the second equality is due to (\ref{HSKim-99}) and Lemma \ref{Ha-Seong-9}, finishing the 7th case. This completes the proof that $\pi_1(T^5/\Gamma, \pi(p_0))$ is trivial.

Now recall that $X$ is obtained by replacing 12 $S^1\times (B_{\frac{1}{10}}/\{\pm 1\})$ with 12 $S^1\times Y$.
Since the zero section $S^2$ is a strong deformation retract of $Y$, $\pi_1(Y)=0$, and hence  $S^1\times (B_{\frac{1}{10}}/\{\pm 1\})$ and $S^1\times Y$ have the same fundamental group $\Bbb Z$. Moreover their boundaries are the same, i.e. $S^1\times \Bbb RP^3$ and the induced inclusion maps on $\pi_1$ are also the same.
Therefore we can conclude from the Seifert-Van Kampen theorem that these 12 replacement processes do not change fundamental groups all the way, and this proves that $X$ is simply connected.

\end{proof}

\begin{proposition}\label{queen}
$H^2(X,\Bbb Z)\simeq \Bbb Z^{13}$.
\end{proposition}
\begin{proof}
To prove $b_2(X)=13$, first recall that $H^{2}(T^5/\Gamma,\Bbb R)$ is isomorphic to the space $\mathcal{H}^2(T^5/\Gamma,\Bbb R)$ of the harmonic 2-forms
on an orbifold $T^5/\Gamma$ with a flat orbifold metric. (The deRham theorem and the Hodge theorem hold true in an orbifold too \cite{Le-orbi}.)
Since the $\Gamma$ action preserves the flat metric of $T^5$, $\mathcal{H}^2(T^5/\Gamma,\Bbb R)$ is just the projection of the $\Gamma$-invariant elements
of $\mathcal{H}^{2}(T^{5},\Bbb R)$. By direct checking the only $\Gamma$-invariant harmonic 2-forms of $T^5$ are constant multiples of $dx_{2}\wedge dx_{3}$,
so $H_2(T^5/\Gamma, \Bbb R)\simeq H^2(T^5/\Gamma, \Bbb R)\simeq \Bbb R$.

In the followings all (co)homology groups are over $\Bbb R$, unless otherwise specified.
First $H_1(\breve{T})\simeq H^1(\breve{T})$ vanishes, since $-j^*$ in the following Mayer-Vietoris sequence is an isomorphism :
$$H^1(T^5/\Gamma)\stackrel{(k^*,l^*)}\longrightarrow H^1(\breve{T})\oplus H^1(U)\stackrel{i^*-j^*}\longrightarrow H^1(\partial\breve{T})$$
$$\{0\}\longrightarrow H^1(\breve{T})\oplus \Bbb R^{12}\longrightarrow \Bbb R^{12}$$ where $i,j,k,l$ are obvious inclusion maps.

We claim that $H_{2}(\breve{T})\simeq H_{2}(T^5/\Gamma)$.
In the Mayer-Vietoris sequence of $T^5/\Gamma$ :
$$H_2(\partial\breve{T})\stackrel{(i_*,-j_*)}\rightarrow H_{2}(\breve{T})\oplus H_2(U) \stackrel{k_*+l_*}\rightarrow H_{2}(T^5/\Gamma) \stackrel{\partial_{*}}\rightarrow
H_{1}(\partial\breve{T})\stackrel{(i_*,-j_*)}\rightarrow H_1(\breve{T})\oplus H_1(U)$$
$$\{0\}\stackrel{(i_*,-j_*)}\longrightarrow H_{2}(\breve{T})\oplus \{0\} \stackrel{k_*+l_*}\longrightarrow H_{2}(T^5/\Gamma) \stackrel{\partial_{*}}\longrightarrow
\Bbb R^{12}\stackrel{(i_*,-j_*)}\longrightarrow \{0\}\oplus \Bbb R^{12}$$ one can see that
$$-j_*:H_{1}(\partial\breve{T})\rightarrow H_1(U)$$ is an isomorphism, so $\partial_*=0$ implying that $$k_*:H_2(\breve{T})\rightarrow H_{2}(T^5/\Gamma)$$ is an isomorphism.

In the Mayer-Vietoris sequence of $X$ :
$$H_2(\partial\breve{T})\stackrel{(i_*,-j_*)}\longrightarrow H_{2}(\breve{T})\oplus H_2(N) \stackrel{k_*+l_*}\rightarrow H_{2}(X) \stackrel{\partial_{*}}\rightarrow
H_{1}(\partial\breve{T})\stackrel{(i_*,-j_*)}\longrightarrow H_1(\breve{T})\oplus H_1(N)$$
$$\{0\}\stackrel{(i_*,-j_*)}\longrightarrow \Bbb R\oplus \Bbb R^{12} \stackrel{k_*+l_*}\longrightarrow H_{2}(X) \stackrel{\partial_{*}}\longrightarrow
\Bbb R^{12}\stackrel{(i_*,-j_*)}\longrightarrow \{0\}\oplus \Bbb R^{12}$$ where $N$ denotes the union of 12 $S^1\times Y$,
$$-j_*:H_{1}(\partial\breve{T})\rightarrow H_1(N)$$ is an isomorphism, so $\partial_{*}=0$ implying that
$$k_*+l_*:H_2(\breve{T})\oplus H_2(N)\rightarrow H_{2}(X)$$ is an isomorphism. Thus we get $b_{2}(X)=13$,
and hence $H^2(X,\Bbb Z)\simeq \Bbb Z^{13}$ is obtained  by the universal coefficient theorem and $H_1(X,\Bbb Z)=\{0\}$.
\end{proof}

\begin{proposition}\label{the-poem}
$X$ is nonspin, i.e. $w_2(X)\ne 0$.
\end{proposition}
\begin{proof}
Assume to the contrary that $X$ is spin. Then so is its subset $\breve{T}$.
By using the global trivialization of the tangent bundle
$$TT^5\simeq T^5\times \Bbb R^5\ \ \textrm{where}\ \ \Bbb R^5=\langle\frac{\partial}{\partial x_1},\cdots,\frac{\partial}{\partial x_5}\rangle,$$
the orthonormal frame bundle $P_{SO}$ of $T^5$ can be trivialized as $T^5\times SO(5)$, so its double cover $P_{Spin}$ is trivialized as $T^5\times Spin(5)$.
The orientation-preserving isometric $\Gamma$ action on $T^5$ obviously lifts to $P_{SO}$.

Let's consider this induced $\Gamma$ action on $P_{SO}$ over $\pi^{-1}(\breve{T})$.
Since $\breve{T}$ is spin, this action must also lift to its double cover $P_{Spin}$ over $\pi^{-1}(\breve{T})$.
We shall show that whatever we take a choice of the lifts for the 3 generators $\alpha, \beta, \gamma$ of $\Gamma=(\Bbb Z_2)^3$ to bundle maps of $P_{Spin}$,
the lifted maps do not satisfy commutativity, which is contradictory.

With respect to the above trivializations, $\alpha_*(x,v)$ of $(x,v)\in T^5\times \Bbb R^5\simeq TT^5$ where $\alpha_*$ denotes the derivative map of $\alpha$ is given by
$(\alpha(x),M_\alpha v)$ where $M_\alpha$ is the matrix
$$\left(
               \begin{array}{ccccc}
                 1 & 0 & 0 & 0 & 0 \\
                 0 & -1 & 0 & 0 & 0 \\
                 0 & 0 & -1 & 0 & 0 \\
                 0 & 0 & 0 & -1 & 0 \\
                 0 & 0 & 0 & 0 & -1 \\
               \end{array}
   \right)$$
and the induced bundle map $\alpha_\star(x,A)$ of $(x,A)\in T^5\times SO(5)\simeq P_{SO}$ is given by $(\alpha(x),M_\alpha A)$.

There are two elements in $\mathfrak{p}^{-1}(M_\alpha)\subset Spin(5)$ where $\mathfrak{p}:Spin(5)\rightarrow SO(5)$ is the double covering map.
Recall that $Spin(n)$ can be expressed as the multiplicative subgroup $$\{v_1\cdot \cdots \cdot v_{2m}| \  v_i\in \Bbb R^n, ||v_i||=1, m\geq 0\}$$ of the Clifford algebra $Cl(n)$. (cf. \cite{morgan})
In this expression, $$\mathfrak{p}^{-1}(M_\alpha)=\{\pm e_2\cdot e_3\cdot e_4\cdot e_5\}$$ where $\{e_1,\cdots,e_5\}$ is the standard orthonormal basis of $\Bbb R^n$.

As a trial, let's take any one of them to be the lift $\hat{M}_\alpha$ of $M_\alpha$, and define the principal bundle map $\alpha_\bullet: P_{Spin}\rightarrow P_{Spin}$ as
$$\alpha_\bullet(x,B)=(\alpha(x),\hat{M}_\alpha B)$$ for $(x,B)\in \pi^{-1}(\breve{T})\times Spin(5)$.
Likewise let's choose $\hat{M}_\beta$ among $$\mathfrak{p}^{-1}(M_\beta)=\{\pm e_1\cdot e_2\cdot e_3\cdot e_5\}$$ where $M_\beta$ is
$$\left(
               \begin{array}{ccccc}
                 -1 & 0 & 0 & 0 & 0 \\
                 0 & -1 & 0 & 0 & 0 \\
                 0 & 0 & -1 & 0 & 0 \\
                 0 & 0 & 0 & 1 & 0 \\
                 0 & 0 & 0 & 0 & -1 \\
               \end{array}
   \right)$$
and define the bundle map $\beta_\bullet$.
Whatever the choice may be, it turns out that $$\hat{M}_\alpha\cdot \hat{M}_\beta=- \hat{M}_\beta\cdot \hat{M}_\alpha,$$
since $$(e_2\cdot e_3\cdot e_4\cdot e_5)\cdot (e_1\cdot e_2\cdot e_3\cdot e_5)=-(e_1\cdot e_2\cdot e_3\cdot e_5)\cdot (e_2\cdot e_3\cdot e_4\cdot e_5).$$
Thus $$\alpha_\bullet\circ\beta_\bullet=-\beta_\bullet\circ\alpha_\bullet,$$
yielding the desired contradiction to that the $\Gamma$ action is lifted to $P_{Spin}$ over $\pi^{-1}(\breve{T})$.
\end{proof}

Smooth closed simply-connected 5-manifolds $M$ are classified up to diffeomorphism by S. Smale \cite{smale} and D. Barden \cite{barden}.
In particular if $H_2(M,\Bbb Z)$ is torsion-free, then such $M$ is completely classified by $k:=b_2(M)$ and $w_2(M)$, namely
$$
M \simeq
\left\{
\begin{array}{ll} S^5\#k(S^2\times S^3)
   &\mbox{if } w_2(M)=0\\
   (S^2\tilde{\times} S^3)\#(k-1)(S^2\times S^3)   &\mbox{if } w_2(M)\ne 0
\end{array}\right.
$$
where $S^2\tilde{\times} S^3$ denotes the\footnote{It is unique up to bundle isomorphism because $\textrm{Diff}(S^3)$ deformation retracts onto $O(4)$ and $\pi_1(SO(4))\simeq \Bbb Z_2$.} nontrivial $S^3$-bundle over $S^2$.
However, $H_2$ of our $X$ has nontrivial torsion.
\begin{proposition}
$H_2(X,\Bbb Z)$ has a nontrivial torsion subgroup.
\end{proposition}
\begin{proof}
To make the argument simple, we prove by contradiction, so assume to the contrary that $H_2(X,\Bbb Z)$ is torsion-free.
Recall from the proof of Proposition \ref{queen} that $H_2(X,\Bbb R)\simeq \Bbb R^{13}$ is generated by $H_2(\breve{T},\Bbb R)\simeq \Bbb R$ and $H_2(N,\Bbb R)\simeq \Bbb R^{12}$.
Let's denote the generator (unique up to $\pm 1$) of the torsion-free part of $H^2(\breve{T},\Bbb Z)$ by $\omega$.
Thus $\omega$ is the pull-back from $H^2(X,\Bbb Z)$ under the inclusion map $k:\breve{T}\rightarrow X$.

Since $H_1(X,\Bbb Z)=\{0\}$ and $H_2(X,\Bbb Z)\simeq \Bbb Z^{13}$,
$$H^2(X,\Bbb Z_2)\simeq (\Bbb Z_2)^{13}\simeq H^2(X,\Bbb Z)\otimes \Bbb Z_2$$ by the universal coefficient theorem.
Hence $w_2(X)\ne 0\in H^2(X,\Bbb Z_2)$ is the mod 2 reduction of a torsion-free integral cohomology class.
In the proof of Proposition \ref{the-poem} we showed that $\breve{T}$ is nonspin, so $w_2(\breve{T})\ne 0$.
Thus $$w_2(\breve{T})=k^*w_2(X)\equiv \omega\ \ \ \textrm{mod}\ 2,$$ and hence $$\pi^*w_2(\breve{T})\ne 0,$$
because $\pi^*\omega$ as a real cohomology is $$[dx_{2}\wedge dx_{3}]\ne 0\in H^2(\pi^{-1}(\breve{T}),\Bbb R)\simeq H^2(T^5,\Bbb R)$$
where $\simeq$ is justified by the following Mayer-Vietoris sequence of real cohomology groups :
$$H^1(\pi^{-1}(U))\oplus H^1(\pi^{-1}(\breve{T}))\rightarrow  H^1(\partial(\pi^{-1}(U))) \stackrel{\delta^*}\rightarrow H^2(T^5)\rightarrow H^2(\pi^{-1}(U))\oplus H^2(\pi^{-1}(\breve{T}))\rightarrow  H^2(\partial(\pi^{-1}(U)))$$
$$\Bbb Z^{48}\oplus (\cdot) \longrightarrow \Bbb Z^{48} \stackrel{\delta^*}\longrightarrow H^2(T^5) \longrightarrow  \{0\}\oplus H^2(\pi^{-1}(\breve{T})) \longrightarrow  \{0\}$$
where $\pi^{-1}(U)$ is the union of 48 copies of $S^1\times B^4$ and all homomorphisms except the connecting homomorphism $\delta^*$ are induced by obvious inclusion maps. (From $\delta^*=0$, the desired isomorphism follows.)

On the other hand $$\pi^*w_2(\breve{T})=w_2(\pi^{-1}(\breve{T}))=0,$$ since $\pi^{-1}(\breve{T})\subset T^5$ is spin. This is the desired contradiction.
\end{proof}
As a consequence of this, $X$ cannot admit an effective $T^3$ action by Oh's theorem \cite{oh}.
However, it turns out that $X$ still admits an $\mathcal{F}$-structure to be explained in the next section.

\section{$\mathcal{F}$-structure on $X$}
An $\mathcal{F}$-structure introduced by Cheeger and Gromov \cite{CG-1, CG-2} generalizes an effective $T^k$-action on a manifold.
\begin{definition}
An $\mathcal{F}$-structure on a smooth manifold $M$ is given by data $(U_i,\hat{U_i},T^{k_i})$ for $i\in I$ with the following conditions:
\begin{enumerate}
    \item $\{U_i|i\in I\}$ is a locally finite open cover of $M$.
    \item Each $\pi_i:\hat{U_i}\mapsto U_i$ is a finite Galois covering with covering group $\Gamma_i$.
    \item Each torus $T^{k_i}$ of dimension $k_i$ acts smoothly and effectively on $\hat{U_i}$ in a $\Gamma_i$-covariant way, i.e.
      there exists a homomorphism $$\Psi_i:\Gamma_i\rightarrow \textrm{Aut}(T^{k_i})$$ so that $$\gamma(\mathfrak{t}x)=\Psi_i(\gamma)(\mathfrak{t})\gamma(x)$$ for any $\gamma\in \Gamma_i,$ $\mathfrak{t}\in T^{k_i},$ and $x\in \hat{U_i}$.
    \item If $U_i\cap U_j \neq \emptyset$, then there is a common covering of $\pi_i^{-1}(U_i\cap U_j)$ and $\pi_j^{-1}(U_i\cap U_j)$ such that it is invariant under the lifted actions of $T^{k_i}$ and $T^{k_j}$, and they commute.
\end{enumerate}
\end{definition}
As a special case, an $\mathcal{F}$-structure is called \emph{polarized} if the torus actions defined on the finite coverings are locally free.
The minimum of the dimensions of the orbits is called the \emph{rank} of the $\mathcal{F}$-structure.

\begin{proposition}
$X$ admits a polarized $\mathcal{F}$-structure.
\end{proposition}
\begin{proof}
G. Paternain and J. Petean \cite{PP-1} showed that well-known manifolds obtained by Kummer-type constructions have $\mathcal{F}$-structures. We shall follow their method.

To define an open covering of $X$ for an $\mathcal{F}$-structure, consider the following open sets in $T^5=\Bbb R^5/\Bbb Z^5$ :
$$W^\alpha(\varepsilon):=\{(x_1,\cdots,x_5)\in T^5| \ ||(x_2,x_3)-(a_2,a_3)||<\varepsilon, a_i=0,\frac{1}{2}, i=2,3\}$$
$$W^\beta(\varepsilon):=\{(x_1,\cdots,x_5)\in T^5| \ ||(x_2,x_3)-(a_2,a_3)||<\varepsilon, a_2=\frac{1}{4},\frac{3}{4}, a_3=0,\frac{1}{2}\}$$
$$W^\gamma(\varepsilon):=\{(x_1,\cdots,x_5)\in T^5| \ ||(x_2,x_3)-(a_2,a_3)||<\varepsilon, a_2=0,\frac{1}{2},  a_3=\frac{1}{4},\frac{3}{4}\}.$$
For $\varepsilon\in (0,\frac{1}{100})$ these sets are disjoint, left invariant under the $\Gamma$ action,
and each of them has 4 connected components which are copies of $T^3\times D^2_\varepsilon$.
Observe that $\pi(W^\alpha(\varepsilon))$, $\pi(W^\beta(\varepsilon))$, and $\pi(W^\gamma(\varepsilon))$ are all connected open subsets of $T^5/\Gamma$.
After the surgery resolving the singularities, they got modified into open subsets of $X$,
which are denoted by $U^\alpha(\varepsilon), U^\beta(\varepsilon), U^\gamma(\varepsilon)$ respectively.

Defining $$V:=T^5-cl(W^\alpha(\frac{\varepsilon}{2})\cup W^\beta(\frac{\varepsilon}{2})\cup W^\gamma(\frac{\varepsilon}{2}))$$ where $cl(\cdot)$ denotes the closure,
$\Gamma$ acts freely on $V$ such that $V/\Gamma$ is an open subset of $T^5/\Gamma$.
Thus $V/\Gamma$ considered as a subset of $X$ and $U^\alpha(\varepsilon), U^\beta(\varepsilon), U^\gamma(\varepsilon)$ give an open covering of $X$,
over which we define an $\mathcal{F}$-structure.

Let's denote the circle actions $x_i\mapsto x_i\pm\theta$ for $\theta\in \Bbb R/\Bbb Z$ on $T^5$ by $A_{i\pm}$.
An important fact is that the $T^3$ action given by $A_{1+}\times A_{4+}\times A_{5+}$ leaves invariant $W^\alpha(\varepsilon), W^\beta(\varepsilon),  W^\gamma(\varepsilon), V$ respectively.

To define an $S^1$ action on $U^\alpha(\varepsilon)$, note that $\langle\beta, \gamma\rangle$ swaps the 4 connected components of $W^\alpha(\varepsilon)$ while $\alpha$ preserves each component,
and denote 4 components of $W^\alpha(\varepsilon)$ by $$W^\alpha_1(\varepsilon):=\{(x_1,\cdots,x_5)\in T^5| \ ||(x_2,x_3)-(0,0)||<\varepsilon\},$$
$$W^\alpha_2(\varepsilon):=\beta(W^\alpha_1(\varepsilon)), W^\alpha_3(\varepsilon):=\gamma(W^\alpha_1(\varepsilon)),  W^\alpha_4(\varepsilon):=\beta\gamma(W^\alpha_1(\varepsilon)).$$
We define an $S^1$-action on $W^\alpha_1(\varepsilon),\cdots$, $W^\alpha_4(\varepsilon)$ by using $A_{1+}$, $A_{1-}$,  $A_{1-}$, $A_{1+}$ respectively.
Then this $S^1$-action on $W^\alpha(\varepsilon)$ is $\Gamma$-equivariant, namely commutes with the $\Gamma$-action. Therefore this action descends to
$\pi(W^\alpha(\varepsilon))$, and hence $U^\alpha(\varepsilon)$, since the surgeries do not affect the first coordinate $x_1$ on which $A_{1\pm}$ act.

In a similar way one can define a circle action on $U^\beta(\varepsilon)$ and $U^\gamma(\varepsilon)$ by using $A_{4\pm}$ and $A_{5\pm}$ respectively.

On $V/\Gamma$ we cannot define a circle action directly, but we put the $T^3=A_{1+}\times A_{4+}\times A_{5+}$ action on its Galois cover $V$. Indeed it can be made $\Gamma$-covariant by using the homomorphism
$$\Psi_V:\Gamma\rightarrow \textrm{Aut}(T^3)$$ given by
$$\Psi_V(\alpha)(t_1,t_2,t_3)=(t_1,-t_2,-t_3)$$ $$\Psi_V(\beta)(t_1,t_2,t_3)=(-t_1,t_2,-t_3)$$  $$\Psi_V(\gamma)(t_1,t_2,t_3)=(-t_1,-t_2,t_3).$$

It remains to check the last condition for $\mathcal{F}$-structure. The overlaps consist of 3 regions $$(V/\Gamma)\cap U^\alpha(\varepsilon),\ (V/\Gamma)\cap U^\beta(\varepsilon),\ (V/\Gamma)\cap U^\gamma(\varepsilon).$$
For the common Galois cover $V\cap W^\alpha(\varepsilon)$ of the 1st one, the lifted actions of $T^3=A_{1+}\times A_{4+}\times A_{5+}$
and $S^1=A_{1\pm}$ obviously commute. Likewise for the 2nd and the 3rd ones.

Finally we have a well-defined $\mathcal{F}$-structure on $X$, and it is polarized, since all actions are locally free.
\end{proof}

By the Cheeger-Gromov theorem \cite{CG-1}, any manifold admitting a polarized $F$-structure collapses with sectional curvature bounded, implying that its minimal volume vanishes. Thus we have
\begin{corollary}
$\textrm{MinVol}(X)=0$.
\end{corollary}

\begin{remark}
The above $\mathcal{F}$-structure on $X$ certainly has positive rank.
Cheeger and Gromov \cite{CG-1} also proved that if an $\mathcal{F}$-structure on a manifold has positive rank then the Euler characteristic of the manifold must be 0.
Indeed our $X$ has zero Euler characteristic.

As a consequence of $\textrm{MinVol}(X)=0$, one can also deduce the vanishing of the minimal entropy and various types of other minimal volumes of $X$.(cf. \cite{PP, sung})
\end{remark}

\section{More examples}

One can use this Kummer-type method to construct other examples with such properties. For example, let the group $\Gamma\equiv\langle \alpha \rangle \oplus
\langle \beta \rangle \oplus \langle \gamma \rangle=(\Bbb Z_2)^3$ act on $T^5$ as follows :
$$\alpha : x\mapsto (x_1,-x_2,-x_3,-x_4,\frac{1}{2}-x_5),$$
$$\beta:x\mapsto (x_1,-x_2,-x_3,-x_4,-x_5),$$
$$\gamma:x\mapsto (-x_1,-x_2,\frac{1}{2}-x_3,-x_4,x_5),$$
$$\alpha\beta:x\mapsto (x_1,x_2,x_3,x_4,\frac{1}{2}+x_5),$$
$$\beta\gamma:x\mapsto (-x_1,x_2,\frac{1}{2}+x_3,x_4,-x_5),$$
$$\alpha\gamma:x\mapsto (-x_1,x_2,\frac{1}{2}+x_3,x_4,\frac{1}{2}-x_5),$$
and
$$\alpha\beta\gamma:x\mapsto (-x_1,-x_2,\frac{1}{2}-x_3,-x_4,\frac{1}{2}+x_5).$$
Again the fixed point sets are 16 copies of $S_\alpha, S_\beta,$ and $S_\gamma$ respectively, which are all disjoint.
As before, $\langle\beta,\gamma\rangle$ acts freely on the set of 16 $S_\alpha$ identifying them into 4 of them,
and similarly $\langle\alpha,\gamma\rangle$ identifies 16 $S_\beta$ to 4 of them.
But both of $\alpha$ and $\beta$ act freely on the set of 16 $S_\gamma$ identifying them into 8 $S_\gamma$,
while $\alpha\beta$ preserves each of $S_\gamma$ acting as a translation $$(p_1,\cdots,p_4,x_5) \mapsto (p_1,\cdots,p_4,\frac{1}{2}+x_5)$$ of a circle $S_\gamma$.
Thus the singular set of $T^5/\Gamma$ is a disjoint union of 8 copies of $S^1$ whose singularity is modelled on $S^1\times \Bbb C^2/\{\pm 1\}$
and 8 copies of $S^1$ whose singularity is modelled on $(\Bbb C^2/\{\pm 1\}\times S^1)/\Bbb Z_2$ which is also diffeomorphic to $S^1\times \Bbb C^2/\{\pm 1\}$.

By the same method as before, one can construct $(X,g)$. Since the circle length of $\pi(S_\gamma)=S^1\times \{\textrm{orbifold point}\}$ is the half of the circle lengths  of $\pi(S_\alpha)$ and $\pi(S_\beta)$, the circle length of the glued $S^1\times Y$ should be correspondingly the half.
The same properties listed in Theorem \ref{main-thm} hold in this example too and they can be proved in the same way.
The only difference is that $b_2(X)=b_3(X)=17$ in this case and the $\Gamma$-invariant part of $H^{2}_{DR}(T^{5})$ is $\{c\ dx_{3}\wedge dx_{4}|c\in \Bbb R\}$.

To prove the existence of a polarized $\mathcal{F}$-structure on $X$, one can use the following disjoint open sets in $T^5=\Bbb R^5/\Bbb Z$ :
$$W^{\alpha\beta}(\varepsilon):=\{(x_1,\cdots,x_5)\in T^5| \ ||(x_3,x_4)-(a_3,a_4)||<\varepsilon,  a_i=0,\frac{1}{2}, i=3,4\}$$
$$W^\gamma(\varepsilon):=\{(x_1,\cdots,x_5)\in T^5| \ ||(x_3,x_4)-(a_3,a_4)||<\varepsilon, a_3=\frac{1}{4},\frac{3}{4}, a_4=0,\frac{1}{2}\}$$
$$V:=T^5-cl(W^{\alpha\beta}(\frac{\varepsilon}{2})\cup W^\gamma(\frac{\varepsilon}{2}))$$
all of which are invariant under the $T^2=A_{1+}\times A_{5+}$ action,
and likewise define an $S^1$ action on $U^{\alpha\beta}(\varepsilon)$ and $U^{\gamma}(\varepsilon)$ by using $A_{1\pm}$ and $A_{5\pm}$ respectively.

\section{Discussions and Questions}

Every smooth closed simply connected 5-manifold $M$ admits a metric of positive scalar curvature by the well-known Gromov-Lawson surgery theorem \cite{GL}.
What about a metric of positive Ricci curvature? If $H_2(M,\Bbb Z)$ is torsion-free, then $M$ admits a metric of positive Ricci curvature \cite{sha},
and even a toric Sasaki-Einstein metric (of positive scalar curvature) under the further assumption that $M$ is spin and $b_2(M)$ is odd \cite{craig}.
It is left as an interesting question whether our constructed manifolds $X$ admit metrics of positive Ricci curvature or not.
Since our $X$ are nonspin, they cannot admit a Sasaki-Einstein structure.

There are other ways of constructing simply connected almost Ricci-flat 5-manifolds. One may also use cylindrical construction as in \cite{koval1}.

It is a difficult task to find out whether these almost-Ricci flat 5-manifolds actually admit a Ricci-flat metric or not.
Should one of them does, it would serve as a much-sought-after example of a compact simply connected Ricci-flat manifold with generic holonomy.(cf. \cite{berger})


\end{document}